\documentclass[a4paper]{amsart}
\usepackage{amsmath}
\usepackage{amssymb}
\usepackage{amsfonts}
\usepackage{pxfonts}
\usepackage[OT2,T1]{fontenc}

\usepackage{pict2e}

\usepackage[
textwidth=14.5cm, 
textheight=21cm,
hmarginratio=1:1,
vmarginratio=1:1]{geometry}

\newtheorem{thm}{Theorem}[section]
\newtheorem{cor}[thm]{Corollary}

\newtheorem{prop}[thm]{Proposition}
\theoremstyle{definition}

\newtheorem{prob}[thm]{Problem}
\theoremstyle{remark}
\newtheorem{rem}[thm]{Remark}

\numberwithin{equation}{section}
\numberwithin{figure}{section}

\newcommand{\diff}{\mathrm{d}}
\newcommand{\C}{{\mathbb C}}
\newcommand{\R}{{\mathbb R}}
\newcommand{\D}{{\mathbb D}}
\newcommand{\Te}{{\mathbb T}}

\newcommand{\imag}{\mathrm{i}}

\newcommand{\hDelta}{\varDelta}
\newcommand{\dbar}{\bar\partial}

\newcommand{\Lop}{{\mathbf L}}

\newcommand{\Itilde}{\tilde I}
\newcommand{\Dop}{\mathbf{D}}
\newcommand{\Iop}{\mathbf{I}}

\newcommand\Rop{\mathbf{R}}
\newcommand\Hop{\mathbf{H}}
\newcommand{\Bop}{\mathbf{B}}

\DeclareMathOperator{\re}{Re}
\DeclareMathOperator{\im}{Im}

\begin{document}

\title{On the uniqueness theorem of Holmgren}

%\author{Alexander Borichev}

%\address{Borichev: Laboratoire d'analyse, topologie, probabilit\'es\\
%CMI, Aix-Marseille Universit\'e\\
%39, rue Fr\'ed\'eric Joliot Curie\\
%F--13453 Marseille CEDEX 13\\
%FRANCE}
%
%\email{borichev@cmi.univ-mrs.fr}

\author{Haakan Hedenmalm}
\address{
Hedenmalm: Department of Mathematics\\
KTH Royal Institute of Technology\\
S--10044 Stockholm\\
Sweden}

\email{haakanh@math.kth.se}

\subjclass[2000]{Primary }
 
\thanks{The research of the author was supported by the G\"oran Gustafsson 
Foundation (KVA) and by Vetenskapsr\aa{}det (VR)}

\begin{abstract} 
We review the classical Cauchy-Kovalevskaya theorem and the related 
uniqueness theorem of Holmgren, in the simple setting of
powers of the Laplacian and a smooth curve segment in the plane.
As a local problem, the Cauchy-Kovalevskaya and Holmgren 
theorems supply a complete answer to the existence and uniqueness issues. 
Here, we consider a \emph{global uniqueness problem} of Holmgren's type. 
Perhaps surprisingly, we obtain a connection with the theory of quadrature 
identities, which demonstrates that rather subtle algebraic properties of 
the curve come into play.
For instance, if $\Omega$ is the interior domain of an ellipse, and $I$
is a proper arc of the ellipse $\partial\Omega$, then there exists a 
nontrivial biharmonic function $u$ in $\Omega$ which vanishes to degree three
on $I$ (i.e., all partial derivatives of $u$ of order $\le2$ vanish on $I$) 
\emph{if and only if} the ellipse is a circle.
\end{abstract}

\maketitle

%\centerline{\em In memory of Boris Korenblum}

\section{Introduction} 

\subsection{Basic notation}
\label{subsec-1.1}
Let 
\[
\Delta:=\frac{\partial^2}{\partial x^2}+
\frac{\partial^2}{\partial y^2},\qquad\diff A(z):=\diff x\diff y,
\]
denote the Laplacian and the area element, respectively. 
Here, $z=x+\imag y$
is the standard decomposition into real and imaginary parts. We let $\C$ denote
the complex plane. We also need the standard complex differential operators
\[
\dbar_z:=\frac{1}{2}\bigg(\frac{\partial}{\partial x}+\imag
\frac{\partial}{\partial y}\bigg),\quad \partial_z:=
\frac{1}{2}\bigg(\frac{\partial}{\partial x}-\imag
\frac{\partial}{\partial y}\bigg),
\]
so that $\Delta$ factors as $\hDelta=4\partial_z\dbar_z$. We sometimes drop
indication of the differentiation variable $z$.
A function $u$  on a domain is \emph{harmonic} if $\Delta u=0$ on the domain.
Similarly, for a positive integer $N$, the function $u$ is 
\emph{$N$-harmonic} if $\Delta^N u=0$ on the domain in question.
  
%, while $\D:=\{z\in\C:|z|<1\}$ and $\Te:=
%\{z\in\C:|z|=1\}$ denote the unit disk and the unit circle, respectively. 

\subsection{The theorems of Cauchy-Kovalevskaya and Holmgren
for powers of the Laplacian} 
Let $\Omega$ be a bounded simply connected domain in the plane $\C$ with 
smooth 
boundary. We let $\partial_n$ denote the operation of taking the normal 
derivative.  For $j=1,2,3,\ldots$, we let $\partial_n^j$ denote the $j$-th 
order normal derivative. Here, we understand those higher derivatives in terms
of higher derivatives of the restriction of the function to the line
normal to the boundary at the given boundary point. We consider the 
Cauchy-Kovalevskaya for powers of the Laplacian $\Delta^N$, where 
$N=1,2,3\ldots$.

\begin{thm}
{\rm(Cauchy-Kovalevskaya)} Suppose $I$ is a real-analytic nontrivial arc of 
$\partial\Omega$. Then if $f_j$, for $j=1,\ldots,2N$, are real-analytic 
functions on $I$, there is a function $u$ with $\Delta^N u=0$ in a (planar) 
neighborhood of $I$,  
having $\partial_n^{j-1} u|_I=f_j$ for $j=1,\ldots,2N$. 
The solution $u$ is unique among the real-analytic functions.
\label{thm-CK1}
\end{thm}

Holmgren's theorem gives uniqueness under less restrictive assumptions on 
the data and the solution.

\begin{thm}
{\rm(Holmgren)} Suppose $I$ is a real-analytic nontrivial arc of 
$\partial\Omega$. Then if $u$ is smooth on a planar neighborhood $\mathcal{O}$
of $I$ and $\Delta^N u=0$ holds on $\mathcal{O}\cap\Omega$, with 
$\partial_n^{j-1}u|_I=0$ for $j=1,\ldots,2N$, then $u(z)\equiv 0$ on 
$\mathcal{O}\cap\Omega$, 
provided that the open set $\mathcal{O}\cap\Omega$ is connected.
\label{thm-H1}
\end{thm}

As local statements, the theorems of Cauchy-Kovalevskaya and Holmgren 
complement each other and supply a complete answer to the relevant existence
and uniqueness issues. However, it is often given that the solution 
$u$ is \emph{global}, that is, it solves $\Delta^N u=0$ throughout $\Omega$.
It is then a reasonable question to ask whether this changes anything. 
For instance, in the context of Holmgren's theorem, may we reduce the 
boundary data information on $I$ while retaining the assertion that $u$ 
vanishes identically?
We may, e.g., choose to require a lower degree of flatness along $I$:
\begin{equation}
\partial_n^{j-1}u|_I=0 \quad\text{for}\quad j=1,\ldots,R,
\label{eq-cond-weaker}
\end{equation}
where $1\le R\le 2N$. We call \eqref{eq-cond-weaker} a condition of
vanishing \emph{sub-Cauchy data}.  

\begin{prob} {\rm(Global Holmgren problem)}
Suppose $u$ is smooth in $\Omega\cup I$ and solves $\Delta^Nu=0$ on 
$\Omega$ and has the flatness given by \eqref{eq-cond-weaker} on $I$, for some 
$R=1,\ldots,2N$. For which values of $R$ does it follow that $u(z)\equiv0$ 
on $\Omega$?
\end{prob}

\noindent{\sc Digression on the global Holmgren problem I.}
When $R=2N$, we see that $u(z)\equiv0$ follows from Holmgren's theorem, by
choosing a suitable sequence of neighborhoods $\mathcal{O}$. Another instance
is when $R=N$ and $I=\partial\Omega$. Indeed, in this case, we recognize in
\eqref{eq-cond-weaker} the vanishing of Dirichlet boundary data for the
equation $\Delta^Nu=0$, which necessarily forces $u(z)\equiv0$ given that
we have a global solution. 
When $R<N$ and $I=\partial\Omega$, it is easy to
add additional smooth non-trivial Dirichlet boundary data to  
\eqref{eq-cond-weaker} and obtain a nontrivial solution to $\Delta^N u=0$ on 
$\Omega$ with \eqref{eq-cond-weaker}. So for $I=\partial\Omega$, we see that
the assumptions imply $u(z)\equiv0$ if and only if $N\le R\le 2N$.  
It remains to analyze the case when $I\ne\partial\Omega$. Then either 
$\partial\Omega\setminus I$ consists of a point, or it is an arc.
When $I\ne\partial\Omega$, we cannot
expect that \eqref{eq-cond-weaker} with $R=N$ will be enough to force $u$ 
to vanish on $\Omega$, Indeed, if $\partial\Omega\setminus I$ is a nontrivial 
arc, we may add nontrivial smooth Dirichlet data on $\partial\Omega\setminus I$
\eqref{eq-cond-weaker} and by solving the Dirichlet problem we obtain a 
nontrivial function $u$ with $\Delta^Nu=0$ on $\Omega$ having 
\eqref{eq-cond-weaker} with $R=N$. Similarly, when $\partial\Omega\setminus I$ 
consists of a single point, we may still obtain a nontrivial solution $u$ 
by supplying distributional Dirichlet boundary data which are supported at that 
single point. So, to have a chance to get uniqueness, we must require that 
$N<R\le 2N$. As the case $R=2N$ follows from Holmgren's theorem, the 
interesting interval is $N<R<2N$. For $N=1$, this interval is \emph{empty}.     
However, for $N>1$ it is nonempty, and the problem becomes interesting.

\noindent{\sc Digression on the global Holmgren problem II.}
Holmgren's theorem has a much wider scope than what is presented here. 
It applies a wide range of linear partial differential equations with 
real-analytic coefficients, provided that the given arc $I$ is 
non-characteristic (see \cite{John1}; we also refer the reader to the 
related work of H\"ormander \cite{Horm1}). So the results obtained here 
suggest that we should replace $\Delta^N$ by a more general
linear partial differential operator and see to what happens in the above 
global Holmgren problem. Naturally, the properties of the given linear partial
differential operator and the geometry of the arc $I$ will both influence 
the the answer.     

%Basically, the only reasonable choices are to assume either that (a) $u|_I=0$ 
%or (b) $\partial_n u=0$. As for (a), this is partial Dirichlet data, and by
%supplying nontrivial data along the remaining boundary $\partial\Omega$
%(assuming that $I\ne\partial\Omega$ of course) we may obtain a nontrivial 
%solution $u$ always. A similar consideration involving the Neumann problem
%shows that with (b) we also have nontrivial solutions for $I\ne\partial\Omega$.
%To obtain an interesting \emph{global Holmgren problem} we apparently 
%need to raise the order of the elliptic operator.  

\subsection{Higher dimensions and nonlinear partial differential 
equations}

The global Holmgren problem makes sense also in $\R^n$, and it is natural
to look for a solution there as well. Moreover, if we think of the global
Holmgren problem as asking for uniqueness of the solution for given (not 
necessarily vanishing) sub-Cauchy data, the problem makes sense also for
non-linear partial differential equations.

We analyze the biharmonic equation in three dimensions in Section 
\ref{sec-3D} with respect to the global Holmgren problem. Along the way,
we obtain a factorization of the biharmonic operator $\Delta^2$ as the
product of two $3\times3$ differential operator matrices which are somewhat 
analogous to the \emph{squares} of the Cauchy-Riemann operators 
$\partial_z,\bar\partial_z$ from the two-dimensional setting.

\subsection{The local Schwarz function of an arc} 
If an arc $I$ is real-analytically smooth, there exists an open neighborhood 
$\mathcal{O}_I$ of the arc and a holomorphic function 
$S_I:\mathcal{O}_I\to\C$ such that $S_I(z)=\bar z$ holds along $I$. This 
function $S_I$ is called the \emph{local Schwarz function}. In fact, the 
existence of a local Schwarz function is equivalent to real-analytic smoothness
of the arc. It is possible to ask only for a so-called one-sided Schwarz 
function, which need not be holomorphic in all of $\mathcal{O}_I$ but only 
in $\mathcal{O}_I\cap\Omega$ (the side which belongs to $\Omega$). Already the
existence of a one-sided Schwarz function is very restrictive on the local
geometry of $I$ \cite{Sakai}. To ensure uniqueness of the local Schwarz 
function $S_I$ (including the one-sided setting), we shall \emph{assume that 
both $\mathcal{O}_I$ and $\mathcal{O}_I\cap\Omega$ are connected open sets}.

\subsection{A condition which gives uniqueness for the global 
Holmgren problem}

As before, we let $\Omega$ be a bounded simply connected domain in the plane.
We have obtained the following criterion. In the statement, ``nontrivial'' 
means ``not identically equal to $0$''. Moreover, as above, \emph{we assume that
the set $\mathcal{O}_I\cap\Omega$} -- the domain of definition of the 
(one-sided) Schwarz function -- \emph{is connected}. 

\begin{thm}
Suppose there exists a nontrivial function $u:\Omega\to\C$ with $\Delta^Nu=0$ 
on $\Omega$, which extends to a $C^{2N-1}$-smooth function on $\Omega\cup I$,
where $I$ is a real-analytic arc of $\partial\Omega$. If $R$ is an integer 
with $N<R\le 2N$, and if $u$ has the flatness given by \eqref{eq-cond-weaker} 
on $I$, then there exists a nontrivial function of the form
\begin{equation}
\Psi(z,w)=\psi_N(z)w^{N-1}+\psi_{N-1}(z)w^{N-2}+\cdots+\psi_1(z),
\label{eq-Phi}
\end{equation}
where each $\psi_j(z)$ is holomorphic in $\Omega$ for $j=1,\ldots,N$, such that
\begin{equation}
\Psi(z,w)=\mathrm{O}(|w-S_I(z)|^{R-N})\quad\text{as}\,\,\, w\to S_I(z),
\label{eq-flat1}
\end{equation}
for $z\in\Omega\cap\mathcal{O}_I$. 
\label{thm-1.1}
\end{thm}

The above theorem asserts that $w=S_I(z)$ is the solution (root) of a polynomial
equation [over the ring of holomorphic functions on $\Omega$]
\begin{equation}
\Psi(z,w)=\psi_N(z)w^{N-1}+\psi_{N-1}(z)w^{N-2}+\cdots+\psi_1(z)=0,
\label{eq-fundeq1}
\end{equation}
and that $\Psi(z,w)$ has the indicated additional flatness along $w=S_I(z)$ 
if $N+1<R$. \emph{An equivalent way to express the flatness condition 
\eqref{eq-flat1} is to say that $w=S_I(z)$ solves simultaneously the system 
of equations}
\begin{equation}
\partial_w^{j-1}\Psi(z,w)=\frac{(N-1)!}{(N-j)!}\psi_N(z)w^{N-j}
%+\frac{(N-2)!}{(N-j-1)!}\psi_{N-1}(z)w^{N-j-1}
+\cdots+(j-1)!\psi_j(z)=0,\qquad j=1,\ldots,R-N.
\label{eq-fundeq2}
\end{equation}
%where $j$ ranges over $j=1,\ldots,R-N$. 
The equation \eqref{eq-fundeq1} results from considering $j=1$ in 
\eqref{eq-fundeq2}. Let $J$, $1\le J\le N$, 
be the largest integer such that the holomorphic function $\psi_J(z)$ is 
nontrivial. Since the expression $\Psi(z,w)$ is nontrivial, such an integer 
$J$ must exist. As a polynomial equation 
in $w$, \eqref{eq-fundeq1} will have at most $J-1$ roots for any fixed
$z\in\Omega$. Counting multiplicities, the number of roots is constant
and equal to $J-1$, for points $z\in\Omega$ where $\psi_J(z)\ne0$.
At the exceptional points where $\psi_J(z)=0$, the number of roots is smaller.
With the possible exception of branch points, where some of the roots coalesce,
the roots define locally well-defined holomorphic functions in $\Omega\setminus
\mathrm{Z}(\psi_J)$, where 
\[
\mathrm{Z}(\psi_J):=\{z\in\Omega:\,\psi_J(z)=0\}.
\]
If we take the system \eqref{eq-fundeq2} into account, we see that $J>R-N$.
Indeed, we may effectively rewrite \eqref{eq-fundeq2} in the form
\begin{equation}
\partial_w^{j-1}\Psi(z,w)=\frac{(J-1)!}{(J-j)!}\psi_J(z)w^{J-j}
+\cdots+(j-1)!\psi_j(z)=0,\qquad j=1,\ldots,R-N,
\label{eq-fundeq3}
\end{equation}
and if $J\le R-N$, we may plug in $j=J$ into \eqref{eq-fundeq3}, which would
result in 
\[
\partial_w^{J-1}\Psi(z,w)=(J-1)!\psi_J(z)=0,
\]
which cannot be solved by $w=S_I(z)$ [except on the zero set 
$\mathrm{Z}(\psi_J)$], a contradiction. We think of \eqref{eq-fundeq1} as
saying that $w=S_I(z)$ is an algebraic expression over the ring of holomorphic 
functions on $\Omega$. In particular, the local Schwarz function $S_I$ extends 
to a multivalued holomorphic function in $\Omega\setminus\mathrm{Z}(\psi_J)$ 
with branch cuts. So in particular $S_I$ makes sense not just on 
$\Omega\cap\mathcal{O}_I$ [this is an interior neighborhood of the arc $I$], 
but more generally in $\Omega\setminus\mathrm{Z}(\psi_J)$, if we allow for 
multivaluedness and branch cuts. The condition that $w=S_I(z)$ solves 
\eqref{eq-fundeq1} is therefore rather restrictive.
To emphasize the implications of the above theorem, we formulate a ``negative
version''.

\begin{cor}
Let $I$ be a real-analytically smooth arc of $\partial\Omega$, and suppose
that $R$ is an integer with $N<R\le 2N$. Suppose in addition that the local 
Schwarz function $S_I$ does not solve the system 
\eqref{eq-fundeq2} on $\mathcal{O}_I\cap\Omega$ for any nontrivial function 
$\Psi(z,w)$ of the form \eqref{eq-Phi}.
Then every function function $u$ on $\Omega$, which extends to a 
$C^{2N-1}$-smooth function on $\Omega\cup I$, with $\Delta^Nu=0$ on $\Omega$
and flatness given by \eqref{eq-cond-weaker} on $I$, must be trivial: 
$u(z)\equiv0$.
\label{cor-1}
\end{cor}

In particular, for $N=2$ and $R=3$, the condition \eqref{eq-flat1}
says that $w=S_I(z)$ solves the \emph{linear equation}
\[
\psi_2(z)w+\psi_1(z)=0,
\]
with solution 
\[
w=S_I(z)=-\frac{\psi_1(z)}{\psi_2(z)},
\]
which expresses a \emph{meromorphic function in} $\Omega$. 
We formulate this conclusion as a corollary.

\begin{cor}
Let $I$ be a real-analytically smooth arc of $\partial\Omega$, and suppose
that the local Schwarz function $S_I$ does not extend to a meromorphic function
on $\Omega$. 
Then every function function $u$ on $\Omega$, which extends to a 
$C^{3}$-smooth function on $\Omega\cup I$, with $\Delta^2u=0$ on $\Omega$
and flatness given by 
\[
u|_I=0,\,\,\,\partial_n u|_I=0,\,\,\,\partial^2_n u|_I=0,
\]
must be trivial: $u(z)\equiv0$.
\label{cor-2}
\end{cor}

\begin{rem}
Corollary \ref{cor-2} should be compared with what can be said in the analogous
situation in three dimensions (see Theorem \ref{thm-4.4} below).
\end{rem}

It is well-known 
that having a local Schwarz function which extends meromorphically to $\Omega$
puts a strong rigidity condition on the arc $I$. 
For instance, if $\Omega$
is the domain interior to an ellipse, and $I$ is any nontrivial arc of 
$\partial\Omega$ [i.e., of positive length], then $S_I$ extends to a 
meromorphic function in $\Omega$ if and only if the ellipse is a circle.  
This means that the Global Holmgren Problem gives uniqueness in this case, with 
$N=2$ and $R=3$, unless the ellipse is circular. We formalize this as a 
corollary.

\begin{cor}
Suppose $\Omega$ is the domain interior to an ellipse, and that $I$ is a 
nontrivial arc of the ellipse $\partial\Omega$. Suppose $u$ is $C^3$-smooth in
$\Omega\cup I$, and $\Delta^2u=0$ on $\Omega$. If $u$ has
\[
u|_I=0,\,\,\,\partial_n u|_I=0,\,\,\,\partial^2_n u|_I=0,
\] 
then $u(z)\equiv0$ unless the ellipse is a circle.
\label{cor-3}
\end{cor}
 
\begin{rem}
The smoothness condition in Theorem \ref{thm-1.1} and Corollary 
\ref{cor-1} is somewhat excessive. 
For instance, in Corollaries \ref{cor-2} and \ref{cor-3}, the 
$C^3$-smoothness assumption may be reduced to $C^2$-smoothness.  
The additional smoothness makes for an easy presentation by avoiding 
technicalities.
\label{rem-1.8}
\end{rem}

\subsection{Meromorphic Schwarz function and construction
of arc-flat biharmonic functions}
Here, we study the necessity of the Schwarz function condition in Corollary
\ref{cor-2}.

\begin{thm}
Suppose $\partial\Omega$ is a $C^\infty$-smooth Jordan curve, and that
%%$\partial\Omega$ is a real-analytically smooth Jordan curve, and
$I\subset\partial\Omega$ is a real-analytically smooth arc, such that 
the complementary arc $\partial\Omega\setminus I$ is nontrivial as well. 
If the local Schwarz function $S_I$ extends to a meromorphic function 
in $\Omega$ with finitely many poles, then there exists a nontrivial 
function $u$ on $\Omega$, which extends $C^\infty$-smoothly to
$\Omega\cup I$, with $\Delta^2u=0$ on $\Omega$ and flatness given 
by  
\[
u|_I=0,\,\,\,\partial_n u|_I=0,\,\,\,\partial^2_n u|_I=0.
\]
\label{thm-2}
\end{thm}

\begin{rem}
When $\Omega=\D$, the open unit disk, the Schwarz function for the boundary
is $S_\Te(z)=1/z$, which is a rational function, and in particular, 
meromorphic in $\D$. So if $I$ is a nontrivial arc of the unit circle 
$\Te=\partial\D$, and $\Te\setminus I$ is a nontrivial arc as well, then 
Theorem \ref{thm-2} tells us that there exists a nontrivial biharmonic 
function $u$ on $\D$ which is $C^\infty$-smooth on $\D\cup I$ and has 
the flatness 
\[
u|_I=0,\,\,\,\partial_n u|_I=0,\,\,\,\partial^2_n u|_I=0.
\]
In this case, an explicit function $u$ can be found, which works for any 
nontrivial arc $I\subset\Te$ with $I\ne\Te$. Indeed, we may use a suitable
rotation of the function
\[
u(z)=\frac{(1-|z|^2)^3}{|1-z|^4},
\]    
which is biharmonic with the required flatness except for a boundary 
singularity at $z=1$. This shows that the circle is exceptional in Corollary
\ref{cor-3}. We should mention here that the above kernel $u(z)$ appeared
possibly for the first time in \cite{AbHe}, and then later in \cite{BoHe}
and \cite{Olo}. Elias Stein pointed out that very similar kernels in the upper 
half plane appear in connection with the theory of \emph{axially symmetric
potentials} \cite{Wei}. 
\end{rem}

\begin{rem}
Corollary \ref{cor-2} and Theorem \ref{thm-2} settle completely the issue
of the Global Holmgren problem for $\Delta^2$ with the flatness condition
\eqref{eq-flat1} [for $R=3$], in the case when the meromorphic extension 
of the Schwarz function $S_I$ to $\Omega$ has finitely many poles.
Most likely this [technical] finiteness condition may be removed.
%to whole boundary curve 
%$\partial\Omega$ is real-analytically smooth. 
Moreover, it seems likely that there should exists an analogue of 
Theorem \ref{thm-2} which applies to $N>2$. More precisely, suppose that 
$N<R\le 2N$,  and that the local Schwarz function $w=S_I(z)$ solves a 
polynomial equation system of equations \eqref{eq-fundeq3}
where the highest order nontrivial coefficient $\psi_{J}(z)$ has only 
finitely many zeros in $\D$, and that $I\subset\partial\Omega$ is a 
nontrivial real-analytically smooth arc whose complementary arc is 
nontrivial as well. 
Then there should exist a nontrivial function $u$ on $\Omega$ which is 
$C^{2N-1}$-smooth on $\Omega\cup I$ with $\Delta^Nu=0$ on $\Omega$ having the
flatness given by \eqref{eq-flat1} on $I$. 
\end{rem}

As a corollary to Corollary \ref{cor-2} and Theorem \ref{thm-2}, we 
obtain a complete resolution for real-analytically smooth boundaries.

\begin{cor}
Suppose $\partial\Omega$ is a real-analytically smooth Jordan curve, and that
%%$\partial\Omega$ is a real-analytically smooth Jordan curve, and
$I\subset\partial\Omega$ is a an arc, such that 
the complementary arc $\partial\Omega\setminus I$ is nontrivial as well. 
Then there exists a nontrivial function $u$ on $\Omega$, which extends 
$C^2$-smoothly to $\Omega\cup I$, with $\Delta^2u=0$ on $\Omega$ and 
flatness given by  
\[
u|_I=0,\,\,\,\partial_n u|_I=0,\,\,\,\partial^2_n u|_I=0,
\]
if and only if the local Schwarz function $S_I$ extends to a meromorphic 
function in $\Omega$.
\label{cor-2.1}
\end{cor}

\begin{rem}
In the context of Corollary \ref{cor-2.1}, the condition that the local Schwarz
function extend to  a meromorphic function in $\Omega$ is the same as asking
that $\Omega$ be a \emph{quadrature domain} (see Subsection \ref{subsec-qd}). 
\end{rem}

\section{The proof of Theorem \ref{thm-1.1} and its corollaries}

\subsection{Almansi expansion}
It is well-known that a function $u$ which is \emph{$N$-harmonic on $\Omega$}, 
that is, has $\Delta^N u=0$ on $\Omega$, has an
\emph{Almansi expansion}
\begin{equation}
u(z)=u_1(z)+|z|^2u_2(z)+\cdots+|z|^{2N-2}u_{N}(z),
\label{eq-alm1}
\end{equation}
where the functions $u_j$ are all harmonic in $\Omega$; the 
``coefficient functions'' $u_j$ are all uniquely determined by the given
function $u$. On the other hand, every function $u$ of the form 
\eqref{eq-alm1}, where the functions $u_j$ are harmonic, is $N$-harmonic.

\begin{proof}[Proof of Theorem \ref{thm-1.1}]
The function $u$ is $N$-harmonic in $\Omega$, and hence it has an Almansi
representation \eqref{eq-alm1}. Next, for $j=1,2,3,\ldots$, we we consider 
the function 
\[
U(z):=\partial_z^Nu(z),
\]
where $\partial_z$ is the complex differentiation operator defined in 
Subsection \ref{subsec-1.1}. From the flatness assumption on $u$, we know that
\begin{equation}
\bar\partial_z^{j-1} U(z)=0,\qquad z\in I,\,\,\,\,j=1,\ldots,R-N.
\label{eq-flat1.01}
\end{equation} 
Since
\[
\bar\partial_z^NU(z)=\bar\partial_z^N\partial_z^N u(x)=4^{-N}\Delta^N u(z)=0,
\qquad z\in\Omega,
\]
the Almansi representation for $U$ has the special form
\[
U(z)=U_1(z)+\bar z U_2(z)+\cdots+\bar z^{N-1}U_{N}(z),
\] 
where the functions $U_j$, $j=0,\ldots,N-1$  are all holomorphic in $\Omega$,
and uniquely determined by the function $U$. As $u$ is assumed $C^{2N-1}$-smooth
on $\Omega\cup I$, the function $U$ is $C^{N-1}$-smooth on $\Omega\cup I$. In
particular, 
\begin{equation}
\bar\partial_z^{j-1}U(z)=\bar\partial_z^{j-1}\sum_{k=1}^{N}\bar z^{k-1}
U_{k}(z)=\sum_{k=j}^{N}\frac{(k-1)!}{(k-j)!}U_{k}(z)
\label{eq-calc1.1}
\end{equation}
is $C^{2N-j}$-smooth on $\Omega\cup I$ for $j=1,\ldots,N$. By plugging in 
$j=N$ into \eqref{eq-calc1.1}, we find that $U_{N}$ is continuous on 
$\Omega\cup I$. Next, if we plug in $j=N-1$, we find that $U_{N-1}$ 
is continuous on $\Omega\cup I$. Proceeding iteratively, we see that all the
functions $U_{k}$ are continuous on $\Omega\cup I$ ($k=1,\ldots,N$).
In terms the Almansi representation for $U$, the condition \eqref{eq-flat1.01} 
reads
\begin{equation}
\bar\partial_z^{j-1} U(z)
=\sum_{k=j}^{N}\frac{(k-1)!}{(k-j)!}\bar z^{k-j}U_{k}(z)
=0,\qquad z\in I,\,\,\,\,j=1,\ldots,R-N.
\label{eq-flat1.02}
\end{equation} 

We now define the function $\Psi(z,w)$. We declare that $\psi_j(z):=U_j(z)$,
so that the function $\Psi(z,w)$ is given by 
\[
\Psi(z,w):=\sum_{k=1}^{N}\psi_k(z)w^{k-1}=\sum_{k=1}^{N}U_k(z)w^{k-1}.
\]
By differentiating iteratively with respect to $w$, we find that
\[
\partial_w^{j-1}\Psi(z,w)=\partial_w^{j-1}\sum_{k=1}^{N}\psi_k(z)w^{k-1}
=\sum_{k=j}^{N}\frac{(k-1)!}{(k-j)!}\psi_{k}(z)w^{k-j}=
\sum_{k=j}^{N}\frac{(k-1)!}{(k-j)!}U_{k}(z)w^{k-j},
\]
so that 
\begin{equation}
\partial_w^{j-1}\Psi(z,w)\big|_{w:=S_I(z)}=
\sum_{k=j}^{N}\frac{(k-1)!}{(k-j)!}U_{k}(z)[S_I(z)]^{k-j},
\label{eq-calc1.2}
\end{equation}
and according to \eqref{eq-flat1.02}, the right hand side expression in 
\eqref{eq-calc1.2} vanishes on the arc $I$ for $j=1,\ldots,R-N$, as 
$S_I(z)=\bar z$ there. 
But the right hand side of \eqref{eq-calc1.2} is holomorphic on 
$\Omega\cap\mathcal{O}_I$ and extends continuously to 
$(\Omega\cup I)\cap\mathcal{O}_I$ and apparently vanishes on
$I$ for $j=1,\ldots, R-N$, so by the boundary uniqueness theorem for 
holomorphic functions (e.g., 
Privalov's theorem), the right hand side of \eqref{eq-calc1.2} must vanish 
on $\Omega\cap\mathcal{O}_I$:
\[
\partial_w^{j-1}\Psi(z,w)\big|_{w:=S_I(z)}=0,\qquad z\in\Omega\cap\mathcal{O}_I,
\,\,\,\,j=1,\ldots,R-N.
\]  
This is the system of equations \eqref{eq-fundeq2}, which by Taylor's formula
is equivalent to the flatness condition \eqref{eq-flat1}. 

It remains to be established that the function $\Psi(z,w)$ is nontrivial. 
Since, by construction, $\Psi(z,\bar z)=U(z)$, it is enough to show that $U$ is 
nontrivial. We know by assumption that $u$ is nontrivial, and that 
$\partial_z^Nu=U$ while $u$ has the flatness \eqref{eq-cond-weaker} along
$I$. If $U$ is trivial, i.e., $U(z)\equiv0$, then $\partial_z^Nu=0$ which
is an elliptic equation of order $N$ and since $R>N$, the flatness 
\eqref{eq-cond-weaker} entails that $u(z)\equiv0$, by Holmgren's theorem.
This contradicts the nontriviality of $u$, and therefore refutes the 
putative assumption that $U$ was trivial. The proof is complete.
\end{proof} 

\begin{proof}[Proof of Corollary \ref{cor-1}]
This is just the negative formulation of Theorem \ref{thm-1.1}.
\end{proof}

\begin{proof}[Proof of Corollary \ref{cor-2}]
In this case where $N=2$, the equation \eqref{eq-fundeq1} is linear, so
by Theorem \ref{thm-1.1} with $N=2$ and $R=3$, the existence of a nontrivial
biharmonic function on $\Omega$ with flatness \eqref{eq-flat1} along $I$
forces the local Schwarz function $S_I$ to extend meromorphically to $\Omega$.  
\end{proof}

\begin{proof}[Proof of Corollary \ref{cor-3}]
It is well-known that the Schwarz function for an non-circular ellipse 
develops a branch
cut along the segment between the focal points (cf. \cite{Dav}, \cite{Sha}),
so it cannot in particular be meromorphic in $\Omega$. So, in view of Corollary
\ref{cor-2}, we must have $u(z)\equiv0$, as claimed.
\end{proof}

\section{Quadrature domains and the construction of arc-flat
biharmonic functions}
\label{sec-constr}

\subsection{Quadrature domains}
\label{subsec-qd}

As before, $\Omega$ is a bounded simply connected domain in $\C$. For the 
moment, we assume
in addition that the boundary $\partial\Omega$ is a real-analytically smooth
Jordan curve. As before, $I\subset \partial\Omega$ is a nontrivial arc.
Then the local Schwarz function $S_I$ extends to a local Schwarz function for
the whole boundary curve; we write $S_{\partial\Omega}$ for the extension.  
In \cite{AhSh}, Aharonov and Shapiro show that in this setting, the following 
two conditions are equivalent:
\smallskip

\noindent{(i)} $\,\,\,$ the Schwarz function $S_{\partial\Omega}$ extends to 
a meromorphic function in $\Omega$, 

\noindent{(ii)} $\,\,$ the domain $\Omega$ is a quadrature domain.
\smallskip

Here, the statement that $\Omega$ is a \emph{quadrature domain} means that for
all harmonic functions $h$ on $\Omega$ that are area-integrable 
($h\in L^1(\Omega)$),   
\[
\int_\Omega h\diff A=\langle h,\alpha\rangle_\Omega,
\]
for some distribution $\alpha$ with \emph{finite support contained inside} 
$\Omega$.
The notation $\langle\cdot,\cdot\rangle_\Omega$ is the dual action which 
extends (to the setting of distributions) the standard integral
\[
\langle f,g\rangle_\Omega=\int_\Omega fg\diff A
\]
when $fg\in L^1(\Omega)$. It was also explained in \cite{AhSh} that the 
conditions (i)-(ii) are equivalent a third condition:
\smallskip

\noindent{(iii)} $\,\,$ any conformal map $\varphi:\D\to\Omega$ [with 
$\varphi(\D)=\Omega$] is a \emph{rational function}.
\smallskip

It is easy to see that the condition (iii) entails that the boundary curve
$\partial\Omega$ is \emph{algebraic}. Let us try to understand why the 
implication (i)$\implies$(iii) holds. So, we assume the Schwarz function 
extends to a meromorphic function in $\Omega$, and form the function
\[
\Psi(\zeta):=
\begin{cases}
S_{\partial\Omega}(\varphi(\zeta)),\qquad \zeta\in\bar\D,
\\
\overline{\varphi(1/\bar\zeta)},\qquad\quad\, \zeta\in\D_e,
\end{cases}
\]
where $\D_e:=\{\zeta\in\C:\,\,|\zeta|>1\}$ is the ``exterior disk'', and
$\varphi$ is any [surjective] conformal map $\D\to\Omega$. By the assumed
real-analyticity of $\partial\Omega$, the conformal map $\varphi$ extends
holomorphically (and conformally) across the circle $\Te=\partial\D$, see,
e.g. \cite{Pom}. In particular, $\Psi(\zeta)$ is well-defined on $\Te$,
and is holomorphic in $\C\setminus\Te$. As the two definitions in $\C\setminus
\Te$ agree [in the limit sense] along $\Te$, Morera's theorem gives that 
$\Psi$ extends holomorphically across $\Te$. But then $\Psi$ is a rational
function, as it has only finitely many poles and is holomorphic everywhere 
else on the Riemann sphere $\C\cup\{\infty\}$. If we put 
\[
\varphi_{\mathrm{ext}}(\zeta):=\overline{\Psi(1/\bar\zeta)},
\]
then $\varphi_{\mathrm{ext}}$ is a rational function, which agrees with 
$\varphi$ on $\D$. This establishes assertion (iii). 
 
\subsection{Real-analytic arcs with one-sided meromorphic Schwarz function}
We return to the previous setting of a real-analytic arc $I\subset\partial
\Omega$, where $\Omega$ is a bounded simply connected domain whose
boundary $\partial\Omega$ is a $C^\infty$-smooth Jordan curve. We shall assume
that the local Schwarz function extends to a meromorphic function in $\Omega$
with finitely many poles.
In this more general setting, the surjective conformal mapping 
$\varphi:\D\to\Omega$ extends analytically across the arc 
$\Itilde:=\varphi^{-1}(I)$:
the extension is given by
\[
\varphi_{\mathrm{ext}}(\zeta):=\overline{S_I\circ\varphi(1/\bar\zeta)},\qquad
\zeta\in\D_e.
\]
The extension is then meromorphic in $\D\cup\D_e\cup \Itilde$, with finitely
many poles; we denote it by $\varphi$ as well. 

\begin{proof}[Proof of Theorem \ref{thm-2}]
We assume for simplicity that the arc $I$ is open, i.e. does not contain its 
endpoints. 
Also, without loss of generality, we may assume that the origin $0$ is in 
$\Omega$.
We let $\varphi:\D\to\Omega$ be a surjective conformal mapping
with $\varphi(0)=0$, 
which by the above argument extends meromorphically to 
$\D\cup\D_e\cup \Itilde$, with finitelywith finitely many poles.
%from \cite{AhSh} is a rational function.
Here, $\Itilde\subset \Te$ be the arc of the circle for which 
$\varphi(\Itilde)=I\subset\partial\Omega$. 
We let $F$ be a the function 
\begin{equation}
F(\zeta,\xi):=\frac{1}{\varphi(\zeta)}
\int_0^\zeta\frac{1+\bar\xi\eta}{1-\bar\xi\eta}\varphi'(\eta)
\diff\eta,
\label{eq-defF1.1}
\end{equation}
where $|\xi|=1$ is assumed. For fixed $\xi\notin \Itilde$, the function 
$F(\cdot,\xi)$ is well-defined and holomorphic in a neighborhood of 
$\D\cup\Itilde$. Moreover, $F(\zeta,\xi)$ enjoys an estimate in terms of
a (radial) function of $|\zeta|$ which is independent of the parameter 
$\xi\in\Te$. 

%\emph{real-valued} harmonic function on $\D$ which extends to a 
%$C^\infty$-smooth function on the closed disk $\bar\D$ with 
%$h|_{\Itilde}=0$. 
%Since by assumption the complementary arc $\Te\setminus\Itilde$ is 
%nontrivial [has positive length], we may assume that $h$ is non-constant. 
%By Schwarzian reflection, the function $h$ extends harmonically across 
%$\Itilde$, with extension is given by
%\[
%h_{\mathrm{ext}}(\zeta)=-h(1/\bar\zeta),\qquad \zeta\in\D_e.
%\] 
%In the rest of the proof, we write $h$ for the extended function, which is
%harmonic on the simply-connected domain $\D\cup\D_e\cup\Itilde$. By 
%simple-connectivity, we may find a holomorphic function $H$ on 
%$\D\cup\D_e\cup\Itilde$ with $\re H=h$. 
%Next, we define
%\begin{equation}
%V_2(\zeta)=\frac{1}{\varphi(\zeta)-\varphi(0)}
%\int_0^\zeta H(\xi)\varphi'(\xi)\diff\xi=\frac{1}{\varphi(\zeta)}
%\int_0^\zeta H(\xi)\varphi'(\xi)\diff\xi,
%\qquad \zeta\in\D,
%\label{eq-V2}
%\end{equation}
%which extends to a holomorphic function to any simply-connected domain which
%does not contain the poles of $\varphi'$. We also need the function $V_1$
%whose derivative is given by
%\begin{equation}
%V_1'(\zeta)=-\varphi'(\zeta)\int_0^\zeta
%\big\{2V_2'(\xi)+\varphi(\xi)[V_2'/\varphi']'(\xi)\big\}
%\overline{\varphi(1/\bar\xi)}\diff\xi,
%\label{eq-V1}
%\end{equation}
%and defines a well-defined holomorphic function in any simply connected domain
%which avoids the poles and zeros of $\varphi'$. We observe that since $h$
%was non-constant, the holomorphic functions $H$ and $V_2$ are
%non-constant as well. 
Next, we let proceed by considering functions real-valued $v_1,v_2$ that are 
harmonic in $\D$ (to be determined shortly), and associate holomorphic 
functions $V_1,V_2$ with $\im V_1(0)=\im V_2(0)=0$ and $\re V_j=v_j$ for 
$j=1,2$.  Then $2\partial_\zeta v_2(\zeta)=V_2'(\zeta)$, for $j=1,2$.
We form the associated function 
\begin{equation}
v(\zeta):=v_1(\zeta)+|\varphi(\zeta)|^2v_2(\zeta).
\label{eq-defv}
\end{equation}
%By \eqref{eq-V2}, we have $(\varphi V_2)'=H\varphi'$, or, which is the same,
%\begin{equation}
%V_2(\zeta)+\frac{\varphi(\zeta)}{\varphi'(\zeta)}V_2'(\zeta)=H(\zeta). 
%\label{eq-DEV2}
%\end{equation}
%Taking real parts while observing that $2\partial_\zeta v_2(\zeta)
%=V_2'(\zeta)$,
%we see from \eqref{eq-DEV2} that 
%\begin{equation}
%v_2(\zeta)+2\re\bigg\{\frac{\varphi(\zeta)}{\varphi'(\zeta)}\partial_\zeta 
%v_2(\zeta)\bigg\}=h(\zeta). 
%\label{eq-DEV2.1}
%\end{equation}
%Moreover, from the definition \eqref{eq-V1} of $V_1'$, we have 
%\begin{equation}
%[V_1'/\varphi']'(\zeta)+\overline{\varphi(1/\bar\zeta)}
%\big\{2V_2'(\zeta)+
%\varphi(\zeta)[V_2'/\varphi']'(\zeta)\big\}=0. 
%\label{eq-DEV1.01}
%\end{equation} 
The functions $v_1,v_2$ are real-valued and harmonic, and we calculate that
\begin{equation}
\Delta v=\Delta[v_1+|\varphi|^2v_2]=\Delta[|\varphi|^2v_2]=4
|\varphi'|^2\bigg\{v_2+2\re\bigg[
\frac{\varphi}{\varphi'}\partial_\zeta v_2\bigg]
\bigg\},
%=4|\varphi'|^2h,
\label{eq-DEV1.02}
\end{equation}
%if we use \eqref{eq-DEV2.1}. An analogous calculation which uses that 
%$2\partial_\zeta v_j(\zeta)=V_j'(\zeta)$ for $j=1,2$ shows that
and
\begin
{equation}
2\partial_\zeta \frac{1}{\varphi'(\zeta)}\partial_\zeta[v(\zeta)]
=[V_1'/\varphi']'(\zeta)+\overline{\varphi(\zeta)}\big\{2V_2'(\zeta)
+\varphi(\zeta)[V_2'/\varphi']'(\zeta)\big\}.
\label{eq-DEV1.03}
\end{equation}
Now, we require of $v_1,v_2$ that the function $v$ gets to have vanishing
second order derivatives along $\Itilde$ in the following sense:
\begin{equation}
\Delta v|_{\Itilde}=0,\quad
\partial_\zeta \frac{1}{\varphi'}\partial_\zeta[v]\bigg|_{\Itilde}=0. 
\label{eq-condv1v2}
\end{equation}
Since 
\[
v_2+2\re\bigg[\frac{\varphi}{\varphi'}\partial_\zeta v_2\bigg]=
\re\bigg\{V_2+\frac{\varphi}{\varphi'}V_2'\bigg\}=
\re\bigg\{\frac{(\varphi V_2')'}{\varphi'}\bigg\},
\]
the first condition in \eqref{eq-condv1v2} may be expressed as
\begin{equation}
\re\bigg\{\frac{(\varphi V_2')'}{\varphi'}\bigg\}=0\quad\text{on}\,\,\, 
\Itilde.
\label{eq-condv1}
\end{equation}
If we let $V_2$ be the holomorphic function with $V_2(0)=0$ whose derivative is
given by
\begin{equation}
V_2'(\zeta)=\int_\Te F(\zeta,\xi)\diff\nu(\xi),
\label{eq-defV2}
\end{equation}
where $F$ is as in \eqref{eq-defF1.1} and $\nu$ is a \emph{real-valued 
Borel measure supported on the complementary arc} $\Te\setminus\Itilde$, then 
condition \eqref{eq-condv1} is automatically met, so that the first 
requirement in \eqref{eq-condv1v2} is satisfied.
It remains to meet the second requirement of \eqref{eq-condv1v2} as well.
In view of \eqref{eq-DEV1.03}, and the uniqueness theorem for holomorphic 
functions, we may write the second requirement in the form
\[
[V_1'/\varphi']'(\zeta)+\bar\varphi(1/\bar\zeta)\big\{2V_2'(\zeta)+
\varphi(\zeta)[V_2'/\varphi']'(\zeta)\big\}=0,
\]
which is the same as
\[
\bigg[\frac{V_1'}{\varphi'}\bigg]'(\zeta)
+\frac{\bar\varphi(1/\bar\zeta)}{\varphi(\zeta)}\,\frac{\diff}{\diff\zeta}
\bigg\{\frac{[\varphi(\zeta)]^2V_2'(\zeta)}{\varphi'(\zeta)}\bigg\}=0.
\]
We think of this is as a second order linear differential equation in $V_1$, 
with a nice holomorphic solution $V_1$ in a neighborhood of $\D\cup\Itilde$ 
unless the finitely many poles in $\D$ of the function 
\[
\frac{\bar\varphi(1/\bar\zeta)}{\varphi(\zeta)}=\frac{S_{I}(\varphi(\zeta))}
{\varphi(\zeta)}
\] 
are felt. In order to suppress those poles, we may ask that the function 
$V_2'$ should have a sufficiently deep zero at each of those poles in $\D$. 
This amounts to asking that
\begin{equation}
V_2^{(j)}(\zeta)=\int_\Te \partial_\zeta^{j-1} F(\zeta,\xi)\diff\nu(\xi)=0
\qquad j=1,\ldots, j_0(\zeta),
\label{eq-finite1}
\end{equation}
for a finite collection of points $\zeta$ in the disk $\D$.  
Taking real and imaginary parts in \eqref{eq-finite1}, we still are left with
a finite number of linear conditions, and the space of real-valued Borel 
measures supported in $\Te\setminus\Itilde$ is infinite-dimensional. So,
clearly, there exists a nontrivial $\nu$ that satisfies \eqref{eq-finite1}. 
If we like, we may even find such a $\nu$ with $C^\infty$-smooth density. 
Then the function $V_2$ is nonconstant, and its real part is nonconstant 
as well.  

Finally, we turn to the issue of the biharmonic function $u$ 
on $\Omega$ that we are looking for. We put 
$\tilde u(z):=v\circ\varphi^{-1}(z)$ and observe that with the choice of 
the Borel measure $\nu$, the function $\tilde u$ is real-valued with
\[
\Delta \tilde u|_I=0,\quad\partial_z^2\tilde u|_I=0,
\]
by \eqref{eq-condv1v2}. This means that all partial derivatives of $\tilde u$
of order $2$ vanish along $I$, which says that both 
$\partial_x\tilde u$ and $\partial_y\tilde u$ have gradient vanishing along
$I$. So both $\partial_x\tilde u$ and $\partial_y\tilde u$ are \emph{constant}
on $I$. If we repeat this argument, we see that there exists an affine 
function $A(z):=A_0+A_1x+A_2y$ such that $u:=\tilde u-A$ has the required
flatness along $I$. Since by construction $\tilde u$ cannot itself be affine, 
this completes the proof of the theorem.
\end{proof}

\begin{proof}[Proof of Corollary \ref{cor-2.1}]
In view of Remark \ref{rem-1.8}, the forward implication follows from
Corollary \ref{cor-2}. In the reverse direction, we appeal to Theorem 
\ref{thm-2} and use the observation that the local Schwarz function $S_I$ is 
automatically holomorphic in a neighborhood of the entire boundary 
$\partial\Omega$, so it can only have finitely many poles in $\Omega$. 
\end{proof}

\section{The biharmonic equation in three dimensions and
the global Holmgren problem}
\label{sec-3D}

\subsection{Matrix-valued differential operators}
In $\C\cong\R^2$, we may identify a complex-valued function $u=u_1+\imag u_2$,
where $u_1,u_2$ are real-valued, with a column vector:
\[
u\sim 
\left(
\begin{array}{c}
u_1
\\
u_2
\\
\end{array}
\right).
\]
In the same fashion, we identify the differential operators $\partial_z$ and 
$\bar\partial_z$ with $2\times2$ matrix-valued differential operators
\[
2\partial_z\sim 
\left(
\begin{array}{cc}
\partial_x &\partial_y
\\
-\partial_y &\partial_x
\\
\end{array}
\right),
\qquad
2\bar\partial_z\sim 
\left(
\begin{array}{cc}
\partial_x &-\partial_y
\\
\partial_y &\partial_x
\\
\end{array}
\right),
\] 
so that
\[
\Delta=4\partial_z\bar\partial_z\sim
\left(
\begin{array}{cc}
\Delta &0
\\
0 &\Delta
\\
\end{array}
\right),
\]
which identifies the Laplacian $\Delta$ with its diagonal lift. Along the
same lines, we see that
\[
4\partial_z^2\sim 
\left(
\begin{array}{cc}
\partial_x^2-\partial_y^2 &2\partial_x\partial_y
\\
-2\partial_x\partial_y &\partial_x^2-\partial_y^2
\\
\end{array}
\right),
\qquad
4\bar\partial_z^2\sim 
\left(
\begin{array}{cc}
\partial_x^2-\partial_y^2 &-2\partial_x\partial_y
\\
2\partial_x\partial_y &\partial_x^2-\partial_y^2
\\
\end{array}
\right),
\] 
and the main identity which we have used in this paper is simply that
\begin{equation}
\left(
\begin{array}{cc}
\partial_x^2-\partial_y^2 &2\partial_x\partial_y
\\
-2\partial_x\partial_y &\partial_x^2-\partial_y^2
\\
\end{array}
\right) 
\left(
\begin{array}{cc}
\partial_x^2-\partial_y^2 &-2\partial_x\partial_y
\\
2\partial_x\partial_y &\partial_x^2-\partial_y^2
\\
\end{array}
\right)=
\left(
\begin{array}{cc}
\Delta^2 &0
\\
0 &\Delta^2
\\
\end{array}
\right).
\label{eq-mainid1.1}
\end{equation}
While it seems unclear what should be the canonical analogue of the 
Cauchy-Riemann operators $\partial_z,\bar\partial_z$ in the three-dimensional 
setting, it turns out to be possible to find suitable analogues of their 
squares! Indeed, there is a three-dimensional analogue of the factorization 
\eqref{eq-mainid1.1}.
%A similar factorization of $\Delta^2$ holds in $\R^3$ as well. 
We write $x=(x_1,x_2,x_3)$ for a point in $\R^3$, and let $\partial_j$ denote 
the partial derivative with respect to $x_j$, for $j=1,2,3$, and let
\[
\Delta:=\partial_1^2+\partial_2^2+\partial_3^2
\]
be the three-dimensional Laplacian. We then define the $3\times3$ matrix-valued 
differential operators
\[
\Lop:=\left(
\begin{array}{ccc}
\partial_1^2-\partial_2^2-\partial_3^2 &2\partial_1\partial_2&
2\partial_1\partial_3
\\
-2\partial_1\partial_2 &\partial_1^2-\partial_2^2+\partial_3^2
&-2\partial_2\partial_3
\\
-2\partial_1\partial_3&-2\partial_2\partial_3&
\partial_1^2+\partial_2^2-\partial_3^2
\end{array}
\right)
\]
and
\[
\Lop':=\left(
\begin{array}{ccc}
\partial_1^2-\partial_2^2-\partial_3^2 &-2\partial_1\partial_2&
-2\partial_1\partial_3
\\
2\partial_1\partial_2 &\partial_1^2-\partial_2^2+\partial_3^2
&-2\partial_2\partial_3
\\
2\partial_1\partial_3&-2\partial_2\partial_3&
\partial_1^2+\partial_2^2-\partial_3^2
\end{array}
\right).
\]

\begin{prop}
The matrix-valued partial differential operators $\Lop,\Lop'$ commute and 
factor the bilaplacian:
\[
\Lop\Lop'=\Lop'\Lop=\left(
\begin{array}{ccc}
\Delta^2 &0&0
\\
0 &\Delta^2&0
\\
0&0&\Delta^2
\\
\end{array}
\right).
\]
\end{prop}

\begin{proof}
We first observe that it enough to check $\Lop\Lop'$ equals the diagonally 
lifted bilaplacian, because $\Lop'\Lop$ amounts to much the same computation
(after all, $\Lop'$ equals $\Lop$ after the change of variables 
$x_1\mapsto-x_1$).
The entry in the $(1,1)$ corner position of the product equals
%\begin{multline*}
\[
(\partial_1^2-\partial_2^2-\partial_3^2)^2+4(\partial_1\partial_2)^2+
4(\partial_1\partial_3)^2=(\partial_1^2+\partial_2^2+\partial_3^2)^2=\Delta^2.
\]
%\end{multline*}
Similarly, the entry in the $(1,2)$ position equals
\[
(\partial_1^2-\partial_2^2-\partial_3^2)(-2\partial_1\partial_2)
+2\partial_1\partial_2(\partial_1^2-\partial_2^2+\partial_3^2)
+2\partial_1\partial_3(-2\partial_2\partial_3)=0,
\]
and the entry in the $(1,3)$ position equals
\[
(\partial_1^2-\partial_2^2-\partial_3^2)(-2\partial_1\partial_3)
+2\partial_1\partial_2(-2\partial_2\partial_3)
+2\partial_1\partial_3(\partial_1^2+\partial_2^2-\partial_3^2)=0.
\]
Furthermore, the entry in the $(2,1)$ position equals
\[
-2\partial_1\partial_2
(\partial_1^2-\partial_2^2-\partial_3^2)
+(\partial_1^2-\partial_2^2+\partial_3^2)
(2\partial_1\partial_2)
-2\partial_2\partial_3
(2\partial_1\partial_3)
=0,
\]
the entry in the $(2,2)$ position equals
\[
-2\partial_1\partial_2
(-2\partial_1\partial_2)
+(\partial_1^2-\partial_2^2+\partial_3^2)^2
-2\partial_2\partial_3
(-2\partial_2\partial_3)
=(\partial_1^2+\partial_2^2+\partial_3^2)^2=\Delta^2,
\]
and the entry in the $(2,3)$ position equals
\[
-2\partial_1\partial_2
(-2\partial_1\partial_3)
+(\partial_1^2-\partial_2^2+\partial_3^2)
(-2\partial_2\partial_3)
-2\partial_2\partial_3(\partial_1^2+\partial_2^2-\partial_3^2)
=0.
\]
Finally, the entry in the $(3,1)$ position equals
\[
-2\partial_1\partial_3(\partial_1^2-\partial_2^2-\partial_3^2)
-2\partial_2\partial_3(2\partial_1\partial_2)
+
(\partial_1^2+\partial_2^2-\partial_3^2)(2\partial_1\partial_3)=0,
\]
the entry in the $(3,2)$ position equals
\[
-2\partial_1\partial_3(-2\partial_1\partial_2)
-2\partial_2\partial_3(\partial_1^2-\partial_2^2+\partial_3^2)
+(\partial_1^2+\partial_2^2-\partial_3^2)(-2\partial_2\partial_3)=0,
\]
and the entry in the $(3,3)$ corner position equals
\[
4(\partial_1\partial_3)^2+4(\partial_2\partial_3)^2
+(\partial_1^2+\partial_2^2-\partial_3^2)^2=
(\partial_1^2+\partial_2^2+\partial_3^2)^2=\Delta^2.
\]
This completes the proof.
\end{proof}

\subsection{An Almansi-type expansion}
We need to have an Almansi-type representation of the biharmonic functions.
We formulate the result in general dimension $n$. We say that the domain
\emph{$\Omega$ is $x_1$-contractive} if $x\in \Omega$ implies that
$(tx_1,x_2,\ldots,x_n)\in\Omega$ for all $t\in[0,1]$. 

\begin{prop}
If $\Omega\subset \R^n$ is convex and $x_1$-contractive, and if 
$u:\Omega\to\R$ is biharmonic, i.e., solves $\Delta^2u=0$, then 
$u(x)=v(x)+x_1w(x)$, where $v,w$ are harmonic in $\Omega$. 
\label{prop-decomp}
\end{prop}

\begin{proof}
By calculation, we have that
\[
\Delta[x_1w]=(\partial_1^2+\cdots+\partial_n^2)[x_1w]=
2\partial_1w+x_1\Delta w,
\]
so that if $w$ is harmonic, $\Delta[x_1w]=2\partial_1w$, and hence, 
$\Delta^2[x_1w]=2\Delta\partial_1w=2\partial_1\Delta w=0$. It is now clear
that any function of the form $u(x)=v(x)+x_1w(x)$, with $v,w$ both harmonic,
is biharmonic. 

We turn to the reverse implication. 
%Without loss of generality, we may assume
%$0\in\Omega$. 
So, we are given a biharmonic function $u$ on $\Omega$, and attempt to find
the two harmonic functions $v,w$ so that $u(x)=v(x)+x_1w(x)$. We first observe
that $h:=\Delta u$ is a harmonic function, and that if $v,w$ exist, we must
have that $h=\Delta[v+x_1w]=\Delta[x_1w]=2\partial_1w$. 
Let $x':=(x_2,\ldots x_n)\in\R^{n-1}$, so that $x=(x_1,x')$. By calculation, 
then,
%\begin{multline*}
\[
\Delta\int_0^{x_1}h(t_1,x')\diff t_1=\partial_1h(x)+
\int_0^{x_1}\Delta'h(t_1,x')\diff t_1=\partial_1h(x)-
\int_0^{x_1}\partial_1^2h(t_1,x')\diff t_1=\partial_1 h(0,x'),
\]
%\end{multline*}
where we used that $h$ was harmonic, and let $\Delta'$ denote the Laplacian
with respect to $x'=(x_2,\ldots,x_n)$. Next, we observe that the slice 
$\Omega':=\Omega\cap(\{0\}\times\R^{n-1})$ is convex, which allows us to apply 
the results of Section 10.6 of \cite{Hormbook} and obtain a solution $F$ to the
Poisson equation $\Delta' F(x')=\partial_1 h(0,x')$ on $\Omega'$. 
We now declare $w$ to be the function
\[
w(x)=w(x_1,x'):=\frac12\bigg\{\int_0^{x_1}h(t_1,x')\diff t_1-F(x')\bigg\},
\]
which is well-defined since $\Omega$ was assumed $x_1$-contractive.
In view of the above calculation, $w$ is harmonic in $\Omega$, and we quickly
see that $2\partial_1 w=h$, so that $\Delta[x_1w]=h$. Finally we put $v:=
u-x_1w$ which is harmonic in $\Omega$ by construction.
\end{proof}

\begin{rem}
If $I\subset\partial\Omega$ is a relatively open patch on the boundary 
$\partial\Omega$ -- which is assumed $C^\infty$-smooth -- and $u$ is $C^4$-smooth
on $\Omega\cup I$, then the above proof produces a decomposition $u=v+x_1w$,
where $v,w$ are harmonic in $\Omega$ and $C^2$-smooth on $\Omega\cup I$.  
\label{rem-smooth1.2}
\end{rem}

\subsection{Application of the matrix-valued differential operators}
We return to three dimensions and assume $u$ is biharmonic in a bounded 
convex domain $\Omega$ is $\R^3$ which is $x_1$-contractive. We assume 
that the boundary $\partial\Omega$ is $C^\infty$-smooth, and that 
$I\subset\partial\Omega$ is a nontrivial open patch. We may lift $u$ to a
vector-valued in the following three ways: 
\[
u^{\langle1\rangle}:=u\oplus0\oplus0= 
\left(
\begin{array}{c}
u
\\
0
\\
0
\\
\end{array}
\right),
\quad
u^{\langle2\rangle}:=0\oplus u\oplus0=
\left(
\begin{array}{c}
0
\\
u
\\
0
\\
\end{array}
\right),
\quad 
u^{\langle3\rangle}:=u\oplus 0\oplus0
=\left(
\begin{array}{c}
0
\\
0
\\
u
\\
\end{array}
\right).
\]
We assume that \emph{all partial derivatives of $u$ of order $\le 2$ vanish 
on} $I$, and that $u$ is $C^4$-smooth on $\Omega\cup I$. 
Since $u$ is biharmonic in $\Omega$, we apply Proposition \ref{prop-decomp}
to decompose $u^{\langle1\rangle},u^{\langle2\rangle},u^{\langle3\rangle}$: 
\[
u^{\langle1\rangle}=v^{\langle1\rangle}+x_1w^{\langle1\rangle},
\quad u^{\langle2\rangle}=v^{\langle2\rangle}+x_1w^{\langle2\rangle},
\quad u^{\langle3\rangle}=v^{\langle3\rangle}+x_1w^{\langle3\rangle},
\]
with obvious interpretation of $v^{\langle j\rangle},w^{\langle j\rangle}$ as 
vector-valued functions. In view of Remark \ref{rem-smooth1.2},
the functions $v,w$ are both $C^2$-smooth in $\Omega\cup I$. Moreover, 
by the flatness assumption on $u$, 
\[
\Lop' [u^{\langle j\rangle}]=\Lop'[v^{\langle j\rangle}]+\Lop'[x_1w^{\langle j\rangle}]
=0\,\,\,\,\text{on}\,\,\,I,\qquad j=1,2,3.
\]
We let $\Rop$ denote the matrix-valued operator
\[
\Rop:=\left(
\begin{array}{ccc}
\partial_1^2 &-\partial_1\partial_2&
-\partial_1\partial_3
\\
\partial_1\partial_2 &-\partial_2^2
&-\partial_2\partial_3
\\
\partial_1\partial_3&-\partial_2\partial_3&
-\partial_3^2
\end{array}
\right),
\]
and observe that $\Lop'[h]=2\Rop[h]$ holds for all harmonic $3$-vectors $h$. 
In a similar fashion, we calculate that 
$\Lop'[x_1h]=2\Dop[h]+2x_1\Rop[h]$ for harmonic $3$-vectors $h$, where $\Dop$
is the matrix-valued differential operator
\[
\Dop:=\left(
\begin{array}{ccc}
\partial_1 &-\partial_2&
-\partial_3
\\
\partial_2 &\partial_1
&0
\\
\partial_3&0&
\partial_1
\end{array}
\right).
\]
In particular, for $j=1,2,3$,
\[
0=\Lop'[u^{\langle j\rangle}]=2\Rop[v^{\langle j\rangle}]
+2\Dop[w^{\langle j\rangle}]+2x_1\Rop[w^{\langle j\rangle}] 
\quad\text{on}\,\,\,I,
\]
which we may write in the form
\begin{equation}
x_1\Rop[w^{\langle j\rangle}]=-\Rop[v^{\langle j\rangle}]
-\Dop[w^{\langle j\rangle}] 
\quad\text{on}\,\,\,I,\,\,\,\text{for}\,\,\,j=1,2,3.
\label{eq-x1solve1}
\end{equation}
Let $\Hop[f]$ be the \emph{Hessian matrix operator}:
\[
\Hop:=\left(
\begin{array}{ccc}
\partial_1^2 &\partial_1\partial_2&
\partial_1\partial_3
\\
\partial_1\partial_2 &\partial_2^2
&\partial_2\partial_3
\\
\partial_1\partial_3&\partial_2\partial_3&
\partial_3^2
\end{array}
\right),
\]  
where the similarity with $\Rop$ is apparent.
The system \eqref{eq-x1solve1} amounts to the $3\times3$ matrix equation
\begin{equation}
x_1\Hop[w]=-\Hop[v]+\Bop[w],
\label{eq-system1.3}
\end{equation}
where
\[
\Bop:=\left(
\begin{array}{ccc}
-\partial_1 &\partial_2&
\partial_3
\\
-\partial_2 &-\partial_1
&0
\\
-\partial_3&0&
-\partial_1
\end{array}
\right).
\]

\begin{thm}
If, in the above setting, the Hessian $\Hop[w]$ is nonsingular on the patch 
$I$, then the matrix field
\[
X_1:=(\Hop[w])^{-1}(-\Hop[v]+\Bop[w]),
\]
defined in $\Omega\cup I$ wherever the Hessian $\Hop[w]$ is nonsingular,
has the property that $X_1=x_1\Iop$ holds on the patch $I$, where $\Iop$ 
denotes the identity $3\times3$ matrix.
\label{thm-4.4}
\end{thm}

Unfortunately, the determinant of the Hessian of a harmonic function 
may vanish identically (see Lewy \cite{Lew}). However, if the determinant 
vanishes then the given harmonic function is rather special, connected with
the theory of minimal surfaces (Lewy \cite{Lew}). Quite possibly the 
system \eqref{eq-system1.3} should give a lot of information anyway also 
in this case. We mention here Lewy's observation that unless the
harmonic function is affine, the corresponding Hessian has rank at least $2$.  

\begin{rem}
Theorem \ref{thm-4.4} is a three-dimensional analogue of Corollary 
\ref{cor-2}. It should be mentioned that part of the assertion of Theorem 
\ref{thm-4.4} is the equality
\[
\nabla[w+\partial_1v]+x_1\nabla[\partial_1v]=0\quad\text{on}\,\,\, I,
\]
which means that
%has the implication that 
$x_1$ multiplied by one harmonic vector field equals another harmonic 
vector field.
%
%equals the \emph{ratio of two harmonic
%vector fields on} $I$. While it is not quite clear precisely how strong
%this condition is, it is an analogue of the two-variable condition we
%found previously. Note also that the requirement is nontrivial, because 
%$\Lop'[u^{\langle 1\rangle}]$ cannot vanish identically (by Holmgren's 
%theorem, once we
%have checked that $I$ is a noncharacteristic surface for $\Lop'$).   
\end{rem}

\section{Acknowledgements}
The author wishes to thank Michael Benedicks, Alexander Borichev, 
Bj\"orn Gustafsson, Alexandru Ionescu, Dmitry Khavinson, Peter Lindqvist,
Vladimir Maz'ya, and Elias Stein  for interesting conversations related 
to the topic of this paper.

\end{document}